\documentclass{amsart}
\usepackage{euscript}
\usepackage{a4wide}
\usepackage[T2A]{fontenc}
\usepackage[utf8]{inputenc}
\usepackage{amsmath}
\usepackage{style}

\title{Resonances sets of Schr\"{o}dinger operators}
\author{Yurii Belov and Pavel Gubkin}

\thanks{The work of PG in Sections 2, 3 and 6 was performed at the Saint Petersburg Leonhard Euler International Mathematical Institute and supported by the Ministry of Science and Higher Education of the Russian Federation (agreement no. 075–15–2025–343). The work of YB of Sections 4 and 5 is supported by the Ministry of Science and Higher Education of the Russian Federation in the framework of a scientific project under agreement No. 075-15-2025-013. YB and PG are winners of the BASIS competitions and would like to thank its sponsors and jury.}

\address{
    \begin{flushleft} 
		Yurii Belov: j\_b\_juri\_belov@mail.ru 
        \vspace{0.1cm}
        \\Department of Mathematics and Computer Science, St. Petersburg State University\\
        14th Line 29b, Vasilyevsky Island, St. Petersburg, Russia, 199178,
\end{flushleft}
}

\address{
	\begin{flushleft}
		Pavel Gubkin: gubkinpavel@pdmi.ras.ru, gubkin.pv@yandex.ru \\
        \vspace{0.1cm}
		 Department of Mathematics and Computer Science, St. Petersburg State University\\
         14th Line 29b, Vasilyevsky Island, St. Petersburg, Russia, 199178\\
         \vspace{0.1cm}
        St. Petersburg Department of Steklov Mathematical Institute, Russian Academy of Sciences\\ 
		Fontanka 27, 191023 St. Petersburg, Russia
	\end{flushleft}
}

\begin{document}
\begin{abstract}
    We prove that resonances of the \Schr operator with compactly supported potential can contain arbitrary subset of the angle $\{z: -\Im z > C |\Re z|\}$ that satisfies \Blaschke condition. We also establish sufficient conditions for the subsets of wider domains.
\end{abstract}
\maketitle

\section{Introduction}
In the present paper we consider one-dimensional \Schr operators $\H_q$ on the half-line $\R_+ = [0, +\infty)$ of the form
\begin{gather*}
    \H_q = -\frac{d^{2}}{dx^{2}} + q(x),
\end{gather*}
with real-valued compactly supported potentials $q\in L^1(\R_+)$. 
The \emph{Jost solution} $\psi(k,x)$ of $\H_q$ is defined as the unique solution of
\begin{gather*}
        -\psi''(k, x) + q(x) \psi (k, x) = k^{2} \psi(k, x),\qquad k\in\Cm, \qquad x\ge0,
\end{gather*}
that satisfies $\psi(k,x) = e^{ikx}$ for all large $x$.
The corresponding \emph{Jost function} is defined by
\begin{gather*}
    w(k) = \psi(k,0),\qquad k\in\Cm.
\end{gather*}
This function is entire, its zeroes in the lower half-plane $\{z\colon \Im z < 0\}$ are called \emph{resonances} of the operator $\H_q$, we denote the set of resonances by $\Res(q)$. Strictly speaking $\Res(q)$ is not a set but a multiset, where the multiplicity of each resonance equals the multiplicity of the corresponding zero. In the free case $q\equiv 0$ we have $\psi(k,x) = e^{ikx}$, $w(k) = 1$ and $\Res(q) = \emptyset$. When $q$ is nontrivial, the set $\Res(q)$ is infinite; together with the corresponding bound states (zeroes of $w$ with positive imaginary part) resonances uniquely determine the potential, see Theorem 1.1 in \cite{korotyaev2004inverse} {and the papers \cite{korotyaev2004stability}, \cite{Marletta2010}, \cite{Marletta2012} devoted to the stability of the inverse spectral problem}. Resonances can also be defined as the poles of the meromorphic continuation of the resolvent of $\H_q$, see Chapter 2 in the book \cite{DyatlovZworksiBook}.
\medskip

The techniques from complex analysis that relate the growth of entire function to the distribution of its zeroes provide valuable information about the location of resonances. For example, 
the classical \Blaschke condition 
\begin{gather}
    \label{eq: blaschke condition}
    \sum_{\lambda \in \Res(q)} \frac{|\Im \lambda|}{1 + |\lambda|^2} < \infty
\end{gather}
holds for every $q$ with compact support.
The counting function 
\begin{gather*}
N(r) = \#\Big\{z\colon w(z) = 0,\; |z| < r,\, \Im z < 0\Big\}
\end{gather*}
of $\Res(q)$ satisfies the relation $N(r) = \frac{2}{\pi}\ell_q r + o(r)$ as $r\to\infty$, where $\ell_q$ is the diameter of the essential support of the potential $q$, see the papers \cite{froese1997asymptotic} and \cite{zworski1987distribution} for the case of operator on the real line and \cite{korotyaev2004inverse} for the half-line. 
Furthermore, for every $\delta > 0$ we have the asymptotic relation 
\begin{gather}
    \label{eq: apart from a set of density zero}
    \#\Big\{z\colon w(z) = 0,\; |z| < r,\, \Arg(z)\notin [-\pi + \delta, -\delta]\Big\} = o(r),\qquad r\to\infty.
\end{gather}
In other words, apart from a
set of density zero, all the resonances of $\H_q$ are contained in arbitrarily small sectors
about the real axis. On the other hand, it is known that resonances cannot lie too close to the real line, see Theorem 3 in \cite{baranov2017branges} and Section 3 in \cite{hitrik1999bounds}: there always exists a logarithmic strip of the form
\begin{gather}
\label{eq: log strip}
    \big\{z\colon -\log(1 + |\Re z|)\le C\Im z\big\}
\end{gather}
that does not contain any resonances. In particular, for every $h > 0$, there is only a finite number of resonances in  the horizontal strip $\{z\colon -h < \Im z < 0\}$. 
\medskip

Another point of interest is the characterization problem of resonances sets.  The results mentioned above give  natural necessary conditions for a set to be a resonances set. Jost functions that correspond to compactly supported potentials are described in \cite{korotyaev2004inverse}, \cite{baranov2017branges}, see Theorem \ref{Thm jost functions} by E.\,Korotyaev below, however the explicit sufficient conditions for the resonances sets are not present in the literature. Resonances in every bounded domain are free parameters, see Theorem 1.2 in \cite{korotyaev2004inverse}. Theorem 3 in \cite{baranov2017branges} states that there exists a potential with infinitely many resonances in the logarithmic strip of the form \eqref{eq: log strip}, 
in Proposition 7 in the paper \cite{zworski1987distribution} M.\,Zworski constructed a smooth potential on the real line 
with infinitely many resonances on the imaginary line. In the present paper we study the following question:  
\begin{gather}
\label{eq: question}
    \textit{Let $\Lambda$ be a set of points in $\{z\colon \Im z < 0\}$. Does there exist $q$ such that $\Lambda\subset \Res(q)$?}
\end{gather}
\begin{Thm}\label{thm: angle}
If $\Lambda$ satisfies the classical \Blaschke condition $\sum_{\lambda\in \Lambda}\frac{|\Im\lambda|}{1 + |\lambda|^2} < \infty$ and for some $C > 0$ we have $\Lambda\subset\{z: -\Im z \geq C |\Re z|\}$ then 
there exists $q\in L^1(\R_+)$ with compact support such that $\Lambda\subset \Res(q)$.
\end{Thm}
\begin{Thm}\label{thm: wider domain} 
Assume that $\Lambda$ is contained in the domain $\{z: -\Im z \geq C | z|^\beta\}$ for some $C > 0$ and $\beta > 0$. If the counting function $n$ of $\Lambda$ satisfies the relation $n(r) \ls r^\alpha$ for some $\alpha < \beta$ then there exists compactly supported $q\in L^1(\R_+)$ such that $\Lambda\subset \Res(q)$.
\end{Thm}
In the next theorem we show that that the previous theorems are sharp in the following sense: the conclusion of Theorem \ref{thm: angle} no longer holds if we replace the angle with any wider domain; the relation from Theorem \ref{thm: wider domain}  between the counting function and the width of the allowed domain is essential.
\begin{Thm}\label{thm non res example} Let $\tau,\rho\colon \R_+\to\R_+$ be two positive unbounded increasing continuous  functions such that $\tau(r) = o(\rho(r))$ as $r \to \infty$. There exists $\Lambda\subset \{z: -\Im z \geq \tau(|z|)\}$ with $n(r)\le \rho(r)$ that 
is not a subset of $\Res(q)$ for any $q\in L^1(\R_+)$ with compact support.
\end{Thm}

Let us finish the introduction with some concluding remarks.
\begin{itemize}
    \item Theorem \ref{thm: wider domain} can be formulated for the domains $\{z: -\Im z \geq \tau(|z|)\}$ when the function $\tau$ satisfies $\tau(r)\gs \log^5 r$. In this case there exists an unbounded $\rho$ such that the assertion $n(t)\le \rho(t)$ implies the existence of the required potential $q$. 
    \item Assertions on the given set $\Lambda$ from Theorem \ref{thm: angle} imply $n(r) = o(r)$ as $r\to\infty$, which is consistent with \eqref{eq: apart from a set of density zero}; in Theorem \ref{thm: wider domain} we also have $n(r)\ls r^\alpha = o(r)$ as $r\to\infty$. On the other hand, we have $N(r) = O(r)$ as $r\to\infty$ for the counting function of $\Res(q)$ hence in both theorems $\Lambda\neq \Res(q)$ and, even more, $\Lambda$ is very small part of $\Res(q)$.
\end{itemize}

\section{Characterization of Jost functions}
Jost functions corresponding to the \Schr operators with compactly supported potentials are entire functions of exponential type. 
Let $\PW_\sigma$ be the classical Paley-Wiener space of entire functions with spectrum in $[-\sigma, \sigma]$ and $\expPWsigma = e^{iz\sigma}\PW_\sigma$. We also fix the notation
\begin{gather*}
    \Cm_+ = \{z\colon \Im z > 0\}, \qquad \ol{\Cm_+} = \{z\colon \Im z \ge 0\},\qquad \Cm_- = \{z\colon \Im z < 0\}.
\end{gather*}
In the present section we use the characterization of Jost functions from the paper \cite{korotyaev2004inverse} to reduce question \eqref{eq: question} to the construction of the suitable entire function.

\begin{Thm}[E. Korotyaev, Theorem 1.1, \cite{korotyaev2004inverse}]\label{Thm jost functions}
    Let $f$ be an entire function without zeroes in the closed half-plane $\ol{\Cm_+}$ such that
    \begin{gather}\label{eq: jost function assertion}
        z (f(z) - 1) \in \Const + \expPWsigma.
    \end{gather}
    Then $f$ is a Jost function of some \Schr operator with $q\in L^1([0,\sigma])$.
\end{Thm}
\begin{Cor}\label{cor: interpolation}
Consider the set $\Lambda \subset \Cm_-$. If there exists $r \in \expPWsigma$ such that $r(0) = 0$ and $r(\lambda) = \lambda$ for all $\lambda\in\Lambda$ then $\Lambda\subset \Res(q)$ for some potential $q\in L^1(\R_+)$ with $\supp q \subset [0,\sigma]$.
\end{Cor}

\begin{proof}
Let $f$ be an entire function defined by $z(f - 1) = -r$. If $z\neq 0$ the equality $f(z) = 0$ is equivalent to $r(z) = z$ hence for every $\lambda\in\Lambda$ we have $f(\lambda) = 0$. Furthermore, $r\in \expPWsigma$ is bounded in $\ol{\Cm_+}$ hence $f$ has only a finite number of zeros in $\ol{\Cm_+}$, denote them by $z_1, \ldots, z_N$.
Fix $N$ arbitrary points $w_1, \ldots, w_N \in \Cm_-$ and define polynomials
\begin{gather*}
    P(z) = (z - z_1)\cdot\ldots\cdot(z - z_n),\qquad  Q(z) = (z - w_1)\cdot\ldots\cdot(z - w_n).
\end{gather*}
The function  $f_1 = f\cdot\frac{Q}{P}$
is entire, has no zeroes in $\ol{\Cm_+}$ and it satisfies $f_1(\lambda) = f(\lambda) = 0$ for every $\lambda\in\Lambda$.
We claim that $f_1$ is a Jost functions corresponding to some potential supported on $[0, \sigma]$. According to Theorem \ref{Thm jost functions}, it suffices to verify
\begin{gather}
    \label{eq: Jost assertion in cor}
    z(f_1 - 1)\in \Const + \expPWsigma.
\end{gather}
Both $P$ and $Q$ are monic with $\deg P = \deg Q = N$ hence there exists $c\in \Cm$ such that 
\begin{gather}
\label{eq: choice of c}
    \deg(z(Q - P) - cP) <  N.
\end{gather}
The equality  $f_1 = f\cdot\frac{Q}{P} = f + f\cdot\frac{Q - P}{P}$ implies
\begin{gather}
\label{eq: z(f_1 - 1) expansion}
    z(f_1 - 1)= z(f - 1) + f\cdot\frac{z(Q - P)}{P} = z(f - 1) + cf + f\cdot\frac{z(Q - P)- cP}{P}.
\end{gather}
Recall that we have $z(f - 1) = -r \in \expPWsigma$. 
The Paley-Wiener theorem gives $r/z\in \expPWsigma$ hence
\begin{gather*}
    f = -\frac{r}{z} + 1 \in \Const + \expPWsigma.
\end{gather*}
By the choice \eqref{eq: choice of c} of the constant $c$, we have
\begin{gather*}
    \left|\frac{z(Q - P) - cP}{P}\right|= O\left(|z|^{-1}\right), \qquad |z|\to\infty.
\end{gather*}
Once again the Paley-Wiener theorem implies that the third term in the right-hand side of \eqref{eq: z(f_1 - 1) expansion} belongs to $\expPWsigma$. Inclusion \eqref{eq: Jost assertion in cor} follows, the proof is concluded.
\end{proof}

\section{Ideas of the proofs}\label{section: idea}

Theorems \ref{thm: angle} and \ref{thm: wider domain} address  question \eqref{eq: question} of whether a given set $\Lambda$ is contained in the resonances set $\Res(q)$ for some compactly supported $q$. 
According to Corollary \ref{cor: interpolation}, positive answer to the latter follows from the existence of
$r \in \expPWsigma$ that satisfies $r(0) = 0$ and $r(\lambda) = \lambda$ for all $\lambda\in\Lambda$. Equivalently, the existence of $g \in \PW_\sigma$ such that 
\begin{gather}
\label{eq: interpolation problem idea}
    g(0) = 0, \qquad g(\lambda) = \lambda e^{-i\sigma \lambda}, \qquad \lambda\in \Lambda.
\end{gather}
Denote $G(z) = z e^{-i\sigma z}$ and let $H$ be entire with $H(\lambda)=0$ for every $\lambda\in\Lambda$.
We want to construct $g$ from \eqref{eq: interpolation problem idea} by the following interpolation formula:
\begin{gather}
\label{eq: first series representation}
g(z) = H(z)\sum_{\lambda\in\Lambda} \frac{G(\lambda)}{(z-\lambda) H'(\lambda)}, \qquad z \in \mathbb{C}.
\end{gather}
When $z = \lambda \in \Lambda$ there is only one non-zero term in the latter sum that gives $g(\lambda) = G(\lambda)$ as required in $\eqref{eq: interpolation problem idea}$. The problem with this approach is that the series may not converge for other $z\in \Cm$. To overcome this obstacle we group elements of $\Lambda$ that are close to each other into disjoint sets $\{\Lambda_n\}_{n\ge 1}$ and try to obtain the convergence in \eqref{eq: first series representation} for the following order of summation:  
\begin{gather}
\label{eq: second series representation}
    g(z) = H(z)\sum_{n\ge 1}\sum_{\lambda\in\Lambda_n} \frac{G(\lambda)}{(z-\lambda) H'(\lambda)}, \qquad z \in \mathbb{C}.
\end{gather}
We draw contours $\Gamma_n$ around each $\Lambda_n$, see Figure \ref{fig:general contours}, and rewrite inner sums in \eqref{eq: second series representation} by the residue theorem:
\begin{gather}
    \label{eq: from sum to contour}
    \sum_{\lambda\in \Lambda_n}\frac{G(\lambda)}{H'(\lambda)(z - \lambda)} = \frac{1}{2\pi i}\oint_{\Gamma_n}\frac{G(w)\, dw}{H(w)(z - w)}.
\end{gather}
The value $|G(w)| = |w|\exp(-\sigma|\Im w|)$ is small when $w\in\Cm_-$ hence the proper estimate on $|H(w)|$  from below 
(e.g., $|H(w)|\gs e^{-(\sigma - \eps)|\Im w|}$) is sufficient to get convergence of the series in \eqref{eq: second series representation}.

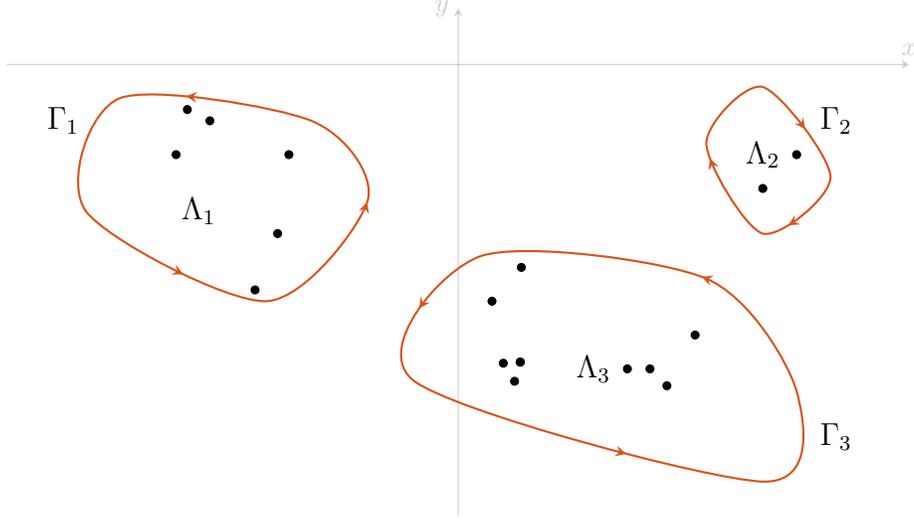
\begin{figure}
\centering
\begin{tikzpicture}[scale =1.5, >=stealth]
\def\pointsize{1pt}

  \draw[opacity = 0.2, thin, ->] (0,3) -- (8,3) node[above] {$x$};
    \draw[opacity = 0.2, thin, ->] (4,-1) -- (4,3.5) node[left] {$y$};

  \tikzset{contour arrows/.style={
    decoration={
      markings,
      mark=at position 0.15 with {\arrow{stealth}},
      mark=at position 0.45 with {\arrow{stealth}},
      mark=at position 0.75 with {\arrow{stealth}}
    },
    postaction={decorate}
  }}

  \draw[myred,thick,contour arrows] plot [smooth cycle] coordinates
    {(1,2.7) (0.7,1.7) (2.3,0.9) (3.2,1.8) (2.7,2.5)};
  \node at (1.7,1.7) {\Large $\Lambda_1$};
  \node at (0.5,2.5) {\Large $\Gamma_1$};

  \draw[myred,thick,contour arrows] plot [smooth cycle] coordinates
    {(6.2,2.3) (6.7,2.8) (7.3,2.0) (6.7,1.5)};
  \node at (6.7,2.2) {\Large $\Lambda_2$};
  \node at (7.35,2.5) {\Large $\Gamma_2$};

  \draw[myred,thick,contour arrows] plot [smooth cycle] coordinates
    {(6.7,-0.7) (7.0,0.1) (6.2,1.1) (4.2,1.3) (3.6,0.2)};
  \node at (5.2,0.3) {\Large $\Lambda_3$};
  \node at (7.35,-0.3) {\Large $\Gamma_3$};

  \filldraw[black] (1.5,2.2) circle (\pointsize);
  \filldraw[black] (1.8,2.5) circle (\pointsize);
  \filldraw[black] (1.6,2.6) circle (\pointsize);
  \filldraw[black] (2.5,2.2) circle (\pointsize);
  \filldraw[black] (2.4,1.5) circle (\pointsize);
  \filldraw[black] (2.2,1.0) circle (\pointsize);

  \filldraw[black] (7.0,2.2) circle (\pointsize);
  \filldraw[black] (6.7,1.9) circle (\pointsize);

  \filldraw[black] (5.7,0.3) circle (\pointsize);
  \filldraw[black] (5.85,0.15) circle (\pointsize);
  \filldraw[black] (6.1,0.6) circle (\pointsize);
  \filldraw[black] (5.5,0.3) circle (\pointsize);
  \filldraw[black] (4.5,0.19) circle (\pointsize);
  \filldraw[black] (4.4,0.35) circle (\pointsize);
  \filldraw[black] (4.55,0.36) circle (\pointsize);
  \filldraw[black] (4.56,1.2   ) circle (\pointsize);
  \filldraw[black] (4.3,0.9) circle (\pointsize);
  \filldraw[black] (4.55,0.36) circle (\pointsize);

\end{tikzpicture}
        \caption{ The set $\Lambda$ is partitioned into the disjoint union $\Lambda = \cup_{n\ge 1}\Lambda_n$ and the contour $\Gamma_n$ is drawn around each of these sets.}
    \label{fig:general contours}
\end{figure}

\section{Points in the angle. Proof of Theorem \ref{thm: angle}}
Recall that the set $\Lambda = \{\lambda_n\}_{n\ge 0}$ satisfies \Blaschke condition in the half-plane and there exists $C > 0$ such that $\Lambda\subset K_C$, where 
\begin{gather*}
    K_C = \{z: -\Im z \geq C |\Re z|\}
\end{gather*}
is the angle in the lower half-plane $\Cm_-$.
Fix arbitrary $\sigma > 0$ and denote $\gamma  = \sigma /3$. There exists $H \in \PW_{\gamma}$ such that $H(0)\neq 0$ and $H(\lambda_n) = 0$ for every $n\ge 0$, see Theorem 2.1.5 in \cite{khabibullin2006completeness}.
Without loss of generality, we may also assume that $\big|[-\gamma, -\gamma + \eps]\cap \supp \hat H\big| > 0$ for all $\eps > 0$.
The Hayman theorem, see Lecture 15 in \cite{levin1996lectures}, states that 
\begin{gather}
\label{eq: H(z) asymptotics}
    \log|H(z)| = -\gamma |\Im z| + o(|z|), \qquad |z| \to \infty,
\end{gather}
holds everywhere in $\Cm_-$ outside of the union of disks of finite view.  In other words, there exists a family of circles 
\begin{gather*}
    B_n = \{z\colon |z - z_n|\le r_n\},\qquad n\ge 1,
\end{gather*}
such that $\sum_{n\ge 1} r_j/|z_j| < \infty$ and \eqref{eq: H(z) asymptotics} holds in the domain $\mathcal{U} = \Cm_-\setminus \cup_{n\ge 1}B_n$.
\bigskip
In particular, for some $A > C$ we have 
\begin{gather*}
    \{z\colon -\Im z = A |\Re z|\}\subset \mathcal{U}.
\end{gather*}
Let $K_A = \{z: -\Im z \geq A |\Re z|\}$ be the corresponding angle. Relation \eqref{eq: H(z) asymptotics} implies the inequality
\begin{gather}
    \label{eq: log H 2 gamma}
    \log|H(z)| \ge -2\gamma|\Im z|
\end{gather}
when $z\in \mathcal{U}\cap K_A$ and $|\Im z|$ is large enough (say $\Im z < -h < 0$).
Furthermore, there exists an unbounded sequence $h \le h_0  < h_1 < h_2 < \ldots$  such that the segments 
\begin{gather*}
    K_A\cup\{z\colon \Im z = - h_k\}, \qquad k\ge 0,
\end{gather*}
are contained in $\mathcal{U}$, see Figure \ref{fig:horizontal lines}. We may also assume that $h _k + 2\le h_{k + 1}\le 2 h_k$ holds for all $k\ge 0$. 
Consider the triangle $Q_0 = K_A\cap \{z\colon \Im z \ge -h\}$, the quadrilaterals
\begin{gather*}
    Q_n = K_A\cup \{z\colon -h_{n}\le \Im z \le -h_{n - 1}\}, \qquad n\ge 1, 
\end{gather*}
and let $\Gamma_n$ be the contour around $Q_n$ for each $n\ge 0$, see Figure \ref{fig:quadrilaterals}. 
\medskip

Let $\Lambda_n$ be the set of zeroes of $H$ inside of $Q_n$. Notice that $\Lambda\subset K_C\subset K_A = \cup_{n\ge 0}Q_n$ hence $\Lambda\subset\cup_{n\ge 0}\Lambda_n$. For every $n\ge 0$, consider the finite sum 
\begin{gather*}
    F_n(z) = \sum_{\lambda\in \Lambda_n}\frac{G(\lambda)}{H'(\lambda)(z - \lambda)}.
\end{gather*}
When $z\notin Q_n$ the latter rewrites, recall \eqref{eq: from sum to contour}, in the form
\begin{gather}
    \label{eq: fn as contour rewritten}
    F_n(z) = \frac{1}{2\pi i}\oint_{\Gamma_n}\frac{G(w)\, dw}{H(w)(z - w)}.
\end{gather}
By the construction, inequality \eqref{eq: log H 2 gamma} holds on $\Gamma_n$ for $n\ge 1$ hence
\begin{gather*}
    |H(w)|^{-1} \le \exp(2\gamma|\Im w|), \qquad \frac{|G(w)|}{|H(w)|}\le |w| \exp((-\sigma+ 2\gamma)|\Im w|),\qquad w\in \Gamma_n.
\end{gather*}
For every $w\in \Gamma_n$ we have $h_n\le |\Im w|$ and $|w|\approx |\Im w| \le h_{n + 1}\approx h_n$. This and the assertion $-\sigma + 2\gamma < 0$ imply
\begin{gather*}
    \frac{|G(w)|}{|H(w)|}\ls h_n \exp((-\sigma+ 2\gamma)h_n), \qquad w\in \Gamma_n.
\end{gather*}
We also know $\len \Gamma_n\ls h_n$ therefore \eqref{eq: fn as contour rewritten} gives
\begin{gather*}
    |F_n(z)|\ls \frac{h_n^2\exp(-(\sigma + 2\gamma)h_n)}{\dist(z, \Gamma_n)}, \qquad z\notin Q_n.
\end{gather*}

\begin{figure}
      \begin{minipage}[t]{0.45\textwidth}
\centering
\begin{tikzpicture}[scale=0.8, >=stealth]

\def\fin{3.75}

\fill[myred!50, fill opacity=1] (1.5,-1) circle (0.4);
\fill[myred!50, fill opacity=1] (0.2,-1.5) circle (0.2);
\fill[myred!50, fill opacity=1] (-2,-1.8) circle (0.3);
\fill[myred!50, fill opacity=1] (2.5,-2.7) circle (0.27);
\fill[myred!50, fill opacity=1] (-1.5,-4) circle (0.3);
\fill[myred!50, fill opacity=1] (1.3,-4) circle (0.21);
\fill[myred!50, fill opacity=1] (-2,-5) circle (0.4);
\fill[myred!50, fill opacity=1] (-3.5,-5) circle (0.2);

\draw[opacity = 0.2, thin, ->] (-\fin,0) -- (\fin,0) node[above] {$x$};
\draw[opacity = 0.2, thin, ->] (0,-5.5) -- (0,0.5) node[left] {$y$};

\coordinate (A) at (0,0);            
\coordinate (B) at (-3,-5);          
\coordinate (C) at (3,-5);           

\draw[thick] (A) -- (B);

\draw[thick] (A) -- (C);

\coordinate (B1) at ($(A)!0.25!(B)$);
\coordinate (C1) at ($(A)!0.25!(C)$);

\coordinate (B2) at ($(A)!0.45!(B)$);
\coordinate (C2) at ($(A)!0.45!(C)$);

\coordinate (B3) at ($(A)!0.65!(B)$);
\coordinate (C3) at ($(A)!0.65!(C)$);

\coordinate (B4) at ($(A)!0.9!(B)$);
\coordinate (C4) at ($(A)!0.9!(C)$);

\coordinate (D0) at (A);
\coordinate (D1) at ($(B1)!0.5!(C1)$);
\coordinate (D2) at ($(B2)!0.5!(C2)$);
\coordinate (D3) at ($(B3)!0.5!(C3)$);
\coordinate (D4) at ($(B4)!0.5!(C4)$);
\coordinate (D) at ($(B)!0.5!(C)$);

\draw[dashed] (B1 -| -\fin,0) -- (B1 -| \fin,0) node[pos=0.1, above] {\small$-h_0$};
\draw[dashed] (B2 -| -\fin,0) -- (B2 -| \fin,0) node[pos=0.1, above] {\small$-h_1$};
\draw[dashed] (B3 -| -\fin,0) -- (B3 -| \fin,0) node[pos=0.1, above] {\small$-h_2$};
\draw[dashed] (B4 -| -\fin,0) -- (B4 -| \fin,0);

\draw (B1) -- (C1);
\draw (B2) -- (C2);
\draw (B3) -- (C3);
\draw (B4) -- (C4);

\end{tikzpicture}
\captionsetup{width=0.9\textwidth}

    \caption{The boundary of the angle $K_A$ and segments $K_A\cap \{\Im z = -h_n\}$ do not intersect the circles (drawn in red), where \eqref{eq: H(z) asymptotics} may fail.}
    \label{fig:horizontal lines}
 \end{minipage}
  \hfill
  \begin{minipage}[t]{0.45\textwidth}
\centering
\begin{tikzpicture}[scale=0.8]

\def\fin{3.75}

\draw[opacity = 0.2, thin, ->] (-\fin,0) -- (\fin,0) node[above] {$x$};
\draw[opacity = 0.2, thin, ->] (0,-5.5) -- (0,0.5) node[left] {$y$};

\coordinate (A) at (0,0);            
\coordinate (B) at (-3,-5);          
\coordinate (C) at (3,-5);           

\draw[thick] (A) -- (B);

\draw[thick] (A) -- (C);

\coordinate (B1) at ($(A)!0.3!(B)$);
\coordinate (C1) at ($(A)!0.3!(C)$);

\coordinate (B2) at ($(A)!0.55!(B)$);
\coordinate (C2) at ($(A)!0.55!(C)$);

\coordinate (B3) at ($(A)!0.85!(B)$);
\coordinate (C3) at ($(A)!0.85!(C)$);

\coordinate (D0) at (A);
\coordinate (D1) at ($(B1)!0.5!(C1)$);
\coordinate (D2) at ($(B2)!0.5!(C2)$);
\coordinate (D3) at ($(B3)!0.5!(C3)$);
\coordinate (D4) at ($(B)!0.5!(C)$);

\draw (B1) -- (C1);
\draw (B2) -- (C2);
\draw (B3) -- (C3);

\coordinate (Q0) at ($(D0)!0.6!(D1)$);
\node at (Q0) {$Q_0$};
\coordinate (firstldots) at ($(D1)!0.5!(D2)$);
\node at (firstldots) {$\ldots$};
\coordinate (center) at ($(D2)!0.5!(D3)$);
\node at (center) {$Q_n$};

\coordinate (center) at ($(D3)!0.5!(D4)$);
\node at (center) {$\ldots$};

\draw[dashed] (B1 -| -\fin,0) -- (B1 -| \fin,0) node[pos=0.1, above] {\small$-h_0$};
\draw[dashed] (B2 -| -\fin,0) -- (B2 -| \fin,0) node[pos=0.125, above] {\small$-h_{n-1}$};
\draw[dashed] (B3 -| -\fin,0) -- (B3 -| \fin,0) node[pos=0.1, above] {\small$-h_n$};

\node at (3,-0.5) { $\Gamma_0$};
  \draw[opacity=0.5, ->] (2.8,-0.55) -- ($(A)!0.4!(C1)$);

\node at (3,-2) { $\Gamma_n$};
  \draw[opacity=0.5, ->] (2.8,-2.2) -- ($(C2)!0.4!(C3)$);

\draw[ultra thick, color=myred, postaction={decorate},
  decoration={markings, mark=at position 0.5 with {\arrow{>}}}]
  (A) -- (B1);
  \draw[ultra thick, color=myred, postaction={decorate},
  decoration={markings, mark=at position 0.5 with {\arrow{>}}}]
  (B1) -- (C1);
  \draw[ultra thick, color=myred, postaction={decorate},
  decoration={markings, mark=at position 0.5 with {\arrow{>}}}]
  (C1) -- (A);

\draw[ultra thick, color=myred, postaction={decorate},
  decoration={markings, mark=at position 0.5 with {\arrow{>}}}]
  (B2) -- (B3);
  \draw[ultra thick, color=myred, postaction={decorate},
  decoration={markings, mark=at position 0.5 with {\arrow{>}}}]
  (B3) -- (C3);
  \draw[ultra thick, color=myred, postaction={decorate},
  decoration={markings, mark=at position 0.5 with {\arrow{>}}}]
  (C3) -- (C2);
  \draw[ultra thick, color=myred, postaction={decorate},
  decoration={markings, mark=at position 0.5 with {\arrow{>}}}]
  (C2) -- (B2);

\end{tikzpicture}

\captionsetup{width=0.9\textwidth}
\caption{The contour $\Gamma_n$ surrounds the polygon $Q_n$ for every $n\ge 0$.}
    \label{fig:quadrilaterals}
    
    \end{minipage}
\end{figure}

\noindent It follows that the series
\begin{gather*}
g_{even}(z) =  \sum_{n\ge 1}H(z)F_{2n}(z)=\sum_{n \geq 1} \sum_{\lambda \in \Lambda_{2n}} \frac{H(z)G(\lambda)}{(z-\lambda) H'(\lambda)}
\end{gather*}
converges for all $z \in \mathbb{C}$ and defines an entire function that satisfies
\begin{gather*}
    |g_{even}(z)|\ls |H(z)|,\qquad z\notin U_{even} = \bigcup_{n\ge 1}\{z\colon \dist(z, Q_{2n})\le 1\}.
\end{gather*}
The structure of $U_{even}$, see Figure \ref{fig:even_nei}, maximum modulus principle and the Paley-Wiener theorem give $g_{even}\in \PW_{2\gamma}$. 
Similarly we deduce $HF_0\in \PW_\gamma$ and 
\begin{gather*}
g_{odd}(z) = \sum_{n\ge 1}H(z)F_{2n - 1}(z)=\sum_{n \geq 1} \sum_{\lambda \in \Lambda_{2n - 1}} \frac{H(z)G(\lambda)}{(z-\lambda) H'(\lambda)} \in \PW_{2\gamma}.
\end{gather*}

\begin{figure}
\centering
\begin{tikzpicture}[scale =1.5, >=stealth]

\def\fin{3.75}

\draw[opacity = 0.2, thin, ->] (-\fin,0) -- (\fin,0) node[above] {$x$};
\draw[opacity = 0.2, thin, ->] (0,-5.5) -- (0,0.5) node[left] {$y$};

\coordinate (A) at (0,0);            
\coordinate (B) at (-3,-5);          
\coordinate (C) at (3,-5);           

\draw[thick] (A) -- (B);

\draw[thick] (A) -- (C);

\coordinate (B1) at ($(A)!0.15!(B)$);
\coordinate (C1) at ($(A)!0.15!(C)$);

\coordinate (B2) at ($(A)!0.25!(B)$);
\coordinate (C2) at ($(A)!0.25!(C)$);

\coordinate (B3) at ($(A)!0.45!(B)$);
\coordinate (C3) at ($(A)!0.45!(C)$);

\coordinate (B4) at ($(A)!0.6!(B)$);
\coordinate (C4) at ($(A)!0.6!(C)$);

\coordinate (B5) at ($(A)!0.75!(B)$);
\coordinate (C5) at ($(A)!0.75!(C)$);

\coordinate (B6) at ($(A)!0.9!(B)$);
\coordinate (C6) at ($(A)!0.9!(C)$);

\coordinate (D0) at (A);
\coordinate (D1) at ($(B1)!0.5!(C1)$);
\coordinate (D2) at ($(B2)!0.5!(C2)$);
\coordinate (D3) at ($(B3)!0.5!(C3)$);
\coordinate (D4) at ($(B4)!0.5!(C4)$);
\coordinate (D5) at ($(B5)!0.5!(C5)$);
\coordinate (D6) at ($(B6)!0.5!(C6)$);
\coordinate (D) at ($(B)!0.5!(C)$);
\def\h{0.2}
\DrawQuadNeigh{B2}{C2}{C1}{B1}{\h}
\DrawQuadNeigh{B4}{C4}{C3}{B3}{\h}
\DrawQuadNeigh{B6}{C6}{C5}{B5}{\h}

\draw[thick] (A) -- (B);

\draw[thick] (A) -- (C);

\draw (B1) -- (C1);
\draw (B2) -- (C2);
\draw (B3) -- (C3);
\draw (B4) -- (C4);
\draw (B5) -- (C5);
\draw (B6) -- (C6);

\coordinate (center) at ($(D0)!0.5!(D1)$);
\node at (center) {$Q_1$};

\coordinate (center) at ($(D1)!0.5!(D2)$);
\node at (center) {$Q_2$};

\coordinate (center) at ($(D3)!0.5!(D4)$);
\node at (center) {$Q_4$};

\coordinate (center) at ($(D5)!0.5!(D6)$);
\node at (center) {$Q_6$};

\end{tikzpicture}

    \caption{The set $U_{even}$ is the union of disjoint neighborhoods of $Q_{2n}$ because of the assertion $h_{n} + 2 \le h_{n + 1}$.}
    \label{fig:even_nei}
\end{figure}
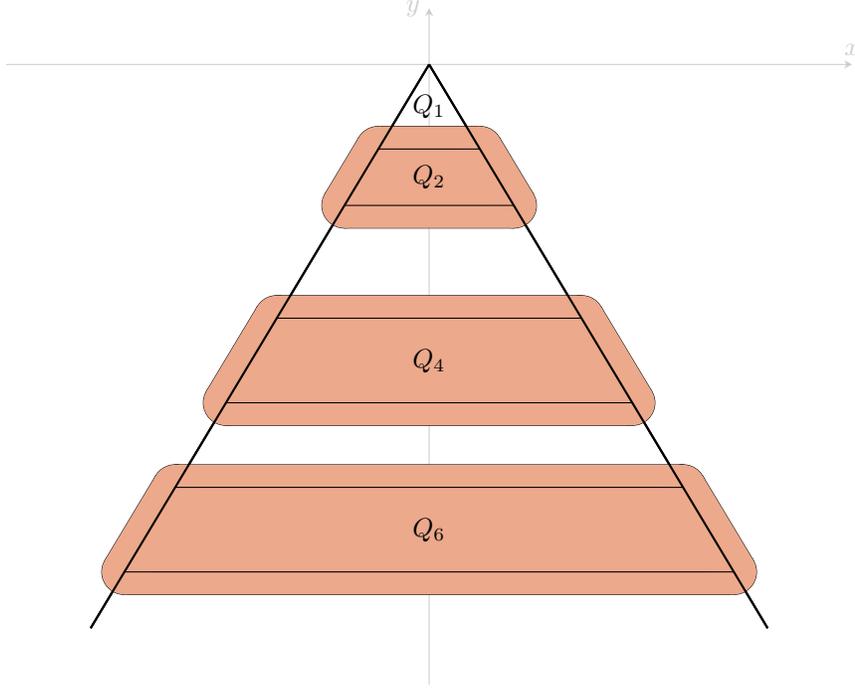

Now the function $g = g_0 + g_{even} + g_{odd}$ almost satisfies all the assertions in \eqref{eq: interpolation problem idea}:
we have $g \in \PW_{2\gamma}\subset\PW_\sigma$ and $g(\lambda) = G(\lambda)$ for all $\lambda \in \Lambda$; however, $g(0) = 0$ may not be satisfied. To overcome this problem, we construct the function $\tilde{g} \in \PW_\sigma $ for the set
\begin{gather*}
\widetilde{\Lambda} = \{ -i \} \cup \{ \lambda - i : \lambda \in \Lambda \}
\end{gather*}
as described above and put $g = \tilde{g}(\cdot + i)$. The proof of Theorem \ref{thm: angle} is finished.

\section{Points in the wider domain. Proof of Theorem \ref{thm: wider domain}}
The proof of Theorem \ref{thm: angle} heavily depends on relation \eqref{eq: H(z) asymptotics} given by the Hayman theorem. Without the assertion $|\Im \lambda|\asymp |\lambda|$ (that is provided by the inclusion $\Lambda\subset K_C$), it is not possible to derive any explicit estimate on $H$ as was done in \eqref{eq: log H 2 gamma}. Thus, different construction is required in the setting of Theorem \ref{thm: wider domain}.

\subsection{Luxemburg-Korevaar Construction}
Let $\Lambda =\{\lambda_n\}_{n\ge 1} $ be the set of points in the complex plane with the counting function $n$ that satisfies the assertion
\begin{gather}
\label{eq: n cond}
\int_0^{\infty} \frac{n(t)}{t^2} \, dt = \sum_{n\ge 1} \frac{1}{|\lambda_n|} < \infty.
\end{gather}
In particular, we have $n(t) = o(t)$ as $t \to \infty$. Define the function 
\begin{gather}
    \label{eq: theta functions def}
    \theta(t) = \int_0^\infty\frac{n(u)}{u}\frac{t^2}{t^2 + u^2}\,du,\qquad t\ge 0,
\end{gather}
and let $w(t)\ge \theta(t) + 2\log t$ be an arbitrary continuous increasing function on $\R_+$. Fix $a > 0$ and for every $k\ge 1$ define $\eps_k$ as the solution of $\psi(1/\eps_k) = k$, where
\begin{gather}
    \label{eq: psi functions def}
    \psi(t) = C \int_a^{t} \frac{w(u)}{u} du,\qquad t\ge a.
\end{gather}
Here $C > 0$ is a constant chosen such that  $\cos(e^{-C}) = 1/e$.
\begin{Thm}[Luxemburg, Korevaar, Theorem 5.2 \cite{luxemburg1971entire}]\label{thm: Luxemburg, Korevaar}
Fix $\sigma > 0$. If the parameter $a$ is large enough then the product
\begin{gather*}
    H(z) = \prod_{n \geq 1} \left( 1 - \frac{z^2}{\lambda_n^2}\right) \prod_{k \geq 1} \cos\left( \eps_k z \right)
\end{gather*}
converges for every $z\in \Cm$ and defines an entire function $H\in \PW_\sigma$.
\end{Thm}

{There exist different ways to construct function from $\PW_\sigma$ with prescribed smallness (or even two-sided estimates) on the real line, see, e.g., Chapters X and XI in the book \cite{KoosisLog2} and the papers \cite{belov, mashnazkhav}}.\\
The following result states that $H(z)$ cannot be very small when $z$ is separated from the zeroes of $H$. We prove it in the next section.
\begin{Thm}\label{thm: H lower bound}
    If $n(t) \ls t^\alpha$ for some $0 <\alpha < 1$ and $a > 0$ is large then the estimate
    \begin{gather*}
    \log|H(z)| \gs -  |z|^{\alpha}\log |z| 
    \end{gather*}
holds for all $z\in \Cm$ that satisfy
\begin{gather*}
    \dist(z, \Lambda) \geq 1,\qquad \dist(-z, \Lambda) \geq 1, \qquad |\Im z|\ge 1.
\end{gather*}
\end{Thm}

\subsection{Proof of Theorem \ref{thm: H lower bound}}
Recall that $n$ is a counting function of $\Lambda$, $\psi$ is given by \eqref{eq: psi functions def} and each $\eps_k$ is the solution of $\psi(1/\eps_k) = k$. The function $\theta$ is defined in \eqref{eq: theta functions def} and $w$ satisfies $w\ge \theta + 2\log t$. Let us introduce two more functions
\begin{gather}
\label{eq: xi function def}
\xi(t) = \int_t^\infty\frac{n(u)}{u^2}\, du = \sum_{|\lambda_k|\ge t}\frac{1}{|\lambda_k|} + \frac{n(t)}{t},\qquad t\ge 0,
\\
\nonumber
\gamma(t) = \int_t^{\infty} \frac{w(u)}{u^2} du, \qquad t\ge 0.
\end{gather}
\begin{Lem}
    For every nonzero $z \in \Cm$ we have
\begin{gather*}
|\cos(z)| \geq \frac{|\Im z|}{2|z|}.
\end{gather*}
\end{Lem}
\begin{proof} Let $z = r e^{i\theta} = r(\cos \theta + i \sin \theta)$. The required inequality is equivalent to $|\cos(z)|\ge |\sin\theta|/2$.
We have $i z = r ( -\sin\theta + i \cos\theta )$ and 
    \begin{gather*}
e^{i z} = e^{r(-\sin\theta + i\cos\theta)} = e^{-r\sin\theta} \left( \cos(r(\cos\theta)) + i \sin(r(\cos\theta)) \right),
\\
e^{-i z} = e^{r(\sin\theta - i\cos\theta)} = e^{r\sin\theta} \left( \cos(r(\cos\theta)) - i \sin(r(\cos\theta)) \right),
\\
2 \cos z = \cos(r \cos\theta) \left( e^{-r \sin\theta} + e^{r \sin\theta} \right)  + i\sin(r \cos\theta)\left( e^{-r \sin\theta} - e^{r \sin\theta} \right).
\end{gather*}
If $|\cos(r\cos\theta)|\ge \sqrt{2}/2$ then we write 
\begin{gather*}
    |\cos z| \ge |\Re\cos z| = \bigg|\frac{\cos(r \cos\theta) \left( e^{-r \sin\theta} + e^{r \sin\theta} \right)}{2}\bigg|\ge |\cos(r \cos\theta)| \ge \frac{\sqrt{2}}{2} \ge \frac{|\sin\theta|}{2}.
\end{gather*}
Otherwise the inequality $|\sin(r\cos\theta)|\ge \sqrt{2}/2$ holds. This implies $r\ge r|\cos\theta|\ge \pi/4$ and
\begin{align*}
    |\cos z| \ge |\Im\cos z| &= \bigg|\frac{\sin(r \cos\theta)\left( e^{-r \sin\theta} - e^{r \sin\theta} \right)}{2}\bigg|
    \\
    &\ge  |\sin(r \cos\theta)|\cdot |r\sin\theta| \ge \frac{\sqrt{2}}{2}\cdot \frac{\pi|\sin\theta|}{4}\ge \frac{|\sin\theta|}{2}.
\end{align*}
The proof is concluded.
\end{proof}
\begin{Lem}\label{Lem: cos estimate}
    For every $ z \in \Cm$ with $|\Im z|\ge 1$, we have
\begin{gather*}
\log \bigg| \prod_{k \geq 1} \cos\left(\eps_k z\right) \bigg| \gs -| z| \cdot \gamma(2| z|) - \psi(2| z|)\cdot \log| z|.
\end{gather*}
\end{Lem}
\begin{proof} Consider the following disjoint sets
\begin{gather*}
    \mathcal{A} = \left\{ k \colon \varepsilon_k \leq \frac{1}{2| z|} \right\}, 
    \qquad 
    \mathcal{B} = \left\{ k \colon \frac{1}{2| z| }< \eps_k\right\}.
\end{gather*}
The set $\mathcal{B}$ is finite, we have $\# \mathcal{B} = \psi(2| z|)$. The previous lemma gives
\begin{gather*}
    \log \left| \prod_{k\in \mathcal{B}} \cos(\varepsilon_k  z) \right|
\geq
\#\mathcal{B}\cdot \log \left( \frac{\operatorname{Im} z}{2| z|} \right)
=
    - \log \frac{2| z|}{\Im z} \cdot
    \psi(2| z|)\gs -\log | z| \cdot
    \psi(2| z|).
\end{gather*}
For every $k\in \mathcal{A}$ we have $|\eps_k z| < 1/2$ hence $\log|\cos(\eps_k z)| \gs -|\eps_k z|$. This implies
\begin{gather*}
\log\left| \prod_{k\in\mathcal{A}} \cos(\varepsilon_k z) \right|
\gs
 -|z| \sum_{k\in \mathcal{A}} \eps_k. 
\end{gather*}
To conclude the proof of the lemma we notice
\begin{gather*}
    \sum_{k\in\mathcal{A}} \eps_k = \sum_{k\ge \psi(2|z|)}\frac{1}{\psi^{-1}(k)} 
    < \int_{\psi(2|z|)}^\infty \frac{dx}{\psi^{-1}(x)} 
    = \int_{2|z|}^\infty\frac{d\psi(x)}{x} 
    \asymp \int_{2|z|}^\infty\frac{w(u)}{u^2} du\ls \gamma(2|z|). 
\end{gather*}
\end{proof}
\begin{Lem}\label{Lem: prod estimate}
    If $z$ satisfies $\dist(z, \Lambda) \geq 1$ and $\dist(-z, \Lambda) \geq 1$ then
\begin{gather*}
\log \bigg| \prod_{k \geq 1} \left( 1 - \frac{z^2}{\lambda_k^2} \right) \bigg| \gs -|z| \cdot \xi(2|z|) - \log|z| \cdot n(2|z|).
\end{gather*}
\end{Lem}
\begin{proof} As in the previous lemma we partition positive integers into the two following sets:
\begin{gather*}
\mathcal{C} = \left\{ k \colon| \lambda_k | \leq 2 | z | \right\}, \qquad
\mathcal{D} = \left\{ k \colon 2 | z | < |\lambda_k| \right\}.
\end{gather*}
The uniform estimate $\log |1 - w|\gs -|w|$ for $w$ with $|w|< 1/2$ gives 
\begin{gather*}
    \log|1 -  z ^2/\lambda_k^2| \gs -| z /\lambda_k|^2 \ge -| z /\lambda_k|, \qquad k\in \mathcal{D}.
\end{gather*}
The latter implies 
\begin{gather*}
    \log \left| \prod_{k\in \mathcal{D}} \left( 1 - \frac{ z ^2}{\lambda_k^2} \right) \right|\gs -\sum_{k\in \mathcal{D}}\frac{| z |}{|\lambda_k|} \ge -| z |\xi(2|\lambda|).
\end{gather*}
On the other hand, for every $k\in \mathcal{C}$ we can use the estimate
\begin{gather*}
    \left| 1 - \frac{ z ^2}{\lambda_k^2} \right| = \frac{|( z  - \lambda_k)( z  + \lambda_k)|}{|\lambda_k|^2} \ge \frac{1}{|\lambda_k|^2} \ge \frac{1}{4| z |^2}. 
\end{gather*}
The set $\mathcal{C}$ is finite, we have $\# \mathcal{C} \le n(2| z |)$ therefore
\begin{gather*}
\log \left| \prod_{k\in \mathcal{C}} \left( 1 - \frac{ z ^2}{\lambda_k^2} \right) \right| \gs -\log(4| z |^2)\cdot \# \mathcal{C}
\gs -\log| z |\cdot n(2| z |).
\end{gather*}
To finish the proof we combine the obtained estimates for products over $\mathcal{C}$ and $\mathcal{D}$. 
\end{proof}

\begin{proof}[Proof of Theorem \ref{thm: H lower bound}]
From the definitions of $\theta$ and $\psi$ we know
\begin{gather*}
    \theta(t) = \int_0^\infty\frac{n(u)}{u}\frac{t^2}{t^2 + u^2}\,du\ls \int_0^\infty u^{\alpha - 1}\frac{t^2}{t^2 + u^2}\,du = t^{\alpha}\int_0^\infty \frac{x^{\alpha - 1}}{1 + x^2}\,dx\ls t^{\alpha},
    \\
    \xi(t) = \int_t^\infty\frac{n(u)}{u^2}\, du \ls  \int_t^\infty u^{\alpha - 2}\, du \ls t^{\alpha - 1}.
\end{gather*}
We can also choose $w(t)\ls t^\alpha$ so that $\psi$ and $\gamma$ satisfy
\begin{gather*}
    \psi(t) \ls \int_1^{t} \frac{w(u)}{u} dt \ls \int_1^{t} u^{\alpha - 1} dt \ls t^\alpha,
    \\
    \gamma(t) = \int_t^{\infty} \frac{w(u)}{u^2} du \ls \int_t^{\infty} u^{\alpha - 2}\, du\ls t^{\alpha - 1}.
\end{gather*}
The required inequality for $H$ now follows from Lemmas \ref{Lem: cos estimate} and \ref{Lem: prod estimate}:
\begin{align*}
    \log|H( z )| &= \log \bigg| \prod_{k \geq 1} \cos\left(\eps_k z \right) \bigg| + \log \bigg| \prod_{n \geq 1} \left( 1 - \frac{ z ^2}{\lambda_n^2} \right) \bigg|
    \\
    &\gs -| z | \cdot \gamma(2| z |) - \psi(2| z |)\cdot \log| z | - | z | \cdot \xi(2| z |) - \log| z | \cdot n(2| z |) \gs -  | z |^{\alpha}\log | z |.
\end{align*}
\end{proof}

\subsection{Points Clustering}\label{section: clustering}
In this section we describe the construction of contours that will be used in the proof of Theorem \ref{thm: wider domain}, recall Figure \ref{fig:general contours}. Fix a number $r > 0$ and let
\begin{gather}
\label{eq: Vr def}
V_r = \big\{z\colon \dist(z, \Lambda)\le r\big\} = \bigcup_{n \geq 1} \big\{z : |\lambda_n - z| \le r \big\}
\end{gather}
be the neighborhood  of $\Lambda$ of the radius $r$.
Denote the connected components of $V_r$ by $U_{r,1}, U_{r,2},\ldots$ and let $\Gamma_{r,n}$ be the contours that run along the boundary of $ U_{r,n}$, see Figure \ref{fig:cluster}. We remark that $\Gamma_{r,n}$ may be disconnected. Let us show that each $U_{r, n}$ is bounded and the length of $\Gamma_n$ is not very large, provided that the set $\Lambda$ satisfies assertion \eqref{eq: n cond}.
\begin{figure}
\centering
\begin{tikzpicture}
  \def\r{0.8}

  \def\redCluster{(-3.2,-3), (-1.8,-2.2), (-2.4,-3.5), (-4.1,-2.5)}
    \def\greenCluster{(1.3,-0.2), (1.6,-1)}
     \def\blueClusterfirst{(-2.5,-3.5), (-1.5,-2.9), (-0.7,-2.4)}
     \def\blueClustersecond{(5.2,-2.5), (6.4,-2.2), (5.2,-0.2), (6.2,-1), (4.2,-1.4)}
  
  \def\rr{0.83}
\tikzset{transparent fill/.style={fill=red, fill opacity=0.4}}

\draw[opacity = 0.2, thin, ->] (-6,2) -- (8,2) node[above] {$x$};
\draw[opacity = 0.2, thin, ->] (1,-5.5) -- (1,2.5) node[left] {$y$};
    
  \foreach \pos in \redCluster {
    \draw[black, ultra thick] \pos circle (\rr);
  }

    \foreach \pos in \redCluster {
    \fill[myred!50, fill opacity=1] \pos circle (\r);
  }

  \foreach \pos in \redCluster {
    \fill[black, thick] \pos circle (0.06);
  }

  \foreach \pos in \blueClustersecond {
    \draw[black, ultra thick] \pos circle (\rr);
  }
  \foreach \pos in \blueClustersecond {
    \fill[myred!50, fill opacity=1] \pos circle (\r);
  }
  \foreach \pos in \blueClustersecond {
    \fill[black, thick] \pos circle (0.06);
  }

  \foreach \pos in \greenCluster {
    \draw[black, ultra thick] \pos circle (\rr);
  }

  \foreach \pos in \greenCluster {
    \fill[myred!50, fill opacity=1] \pos circle (\r);
  }

  \foreach \pos in \greenCluster {
    \fill[black, thick] \pos circle (0.06);
  }

\draw[ultra thick, ->] (1.3,-0.2 +\rr  ) arc (90:130:\rr);
\draw[ultra thick, ->] (1.6 - \rr,-1) arc (180:270:\rr);

\draw[ultra thick, ->] (-1.8+\rr,-2.2) arc (0:75:\rr);
\draw[ultra thick, ->] (-2.4,-3.5 - \rr) arc (-90:-70:\rr);
\draw[ultra thick, ->] (-4.1-\rr,-2.5) arc (180:230:\rr);

\draw[ultra thick, ->] (5.2,-2.5 -\rr) arc (-90:-70:\rr);
\draw[ultra thick, ->] (5.2,-2.5 +\rr) arc (90:91:\rr);

\draw[ultra thick, ->] (5.2,-0.2 -\rr) arc (-90:-85:\rr);

\draw[ultra thick, ->] (6.2 + \rr,-1) arc (0:10:\rr);

\draw[ultra thick, ->] (4.2 - \rr,-1.4) arc (180:181:\rr);


\node at (-5.5,0) { \Large$\Gamma_2$};
\draw[opacity=0.5, ->] (-5.5,-0.2) --  (-4.5,-1.6);
\node at (-4.3,0) { \Large$U_2$};
\draw[opacity=0.5, ->] (-4.25,-0.2) --  (-2.8,-2.8);

\node at (3,-4.5) { \Large$\Gamma_1$};
\draw[opacity=0.5, ->] (2.9,-4.2) --  (3.6,-2.1);
\draw[opacity=0.5, ->] (3,-4.2) --  (4.9,-1.7);
\node at (4,-4.5) { \Large$U_1$};
\draw[opacity=0.5, ->] (4.,-4.2) --  (5,-2.8);

\node at (-1.2,1) { \Large$\Gamma_0$};
\draw[opacity=0.5, ->] (-5.5,-0.2) --  (-4.5,-1.6);
\node at (-1.2, 0.) { \Large$U_0$};
\draw[opacity=0.5, ->] (-0.9,0.9) --  (0.45,0.1);
\draw[opacity=0.5, ->] (-0.9,-0.1) --  (1.1,-0.6);

\end{tikzpicture}
    \caption{Clusters visualization}
    \label{fig:cluster}
\end{figure}
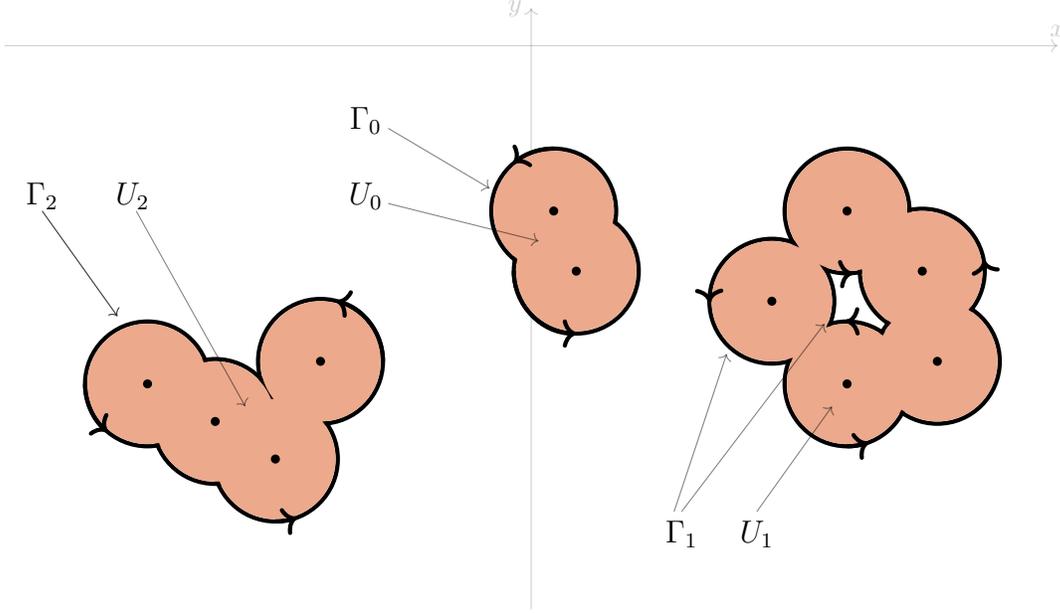

By the construction we have $\Lambda \subset V_r = \cup_{n\ge 1} U_{r, n}$. For every $n\ge 1$ denote $\Lambda_{r,n} = \Lambda \cap U_{r,n}$ and let $\{\lambda_{r,n,0}, \lambda_{r,n,1},\ldots\}$ be the elements of $\Lambda_{r,n}$ enumerated such that $|\lambda_{r,n,0}| \le |\lambda_{r,n,1}|\le\ldots$ holds. Assume that the size of $\Lambda_{r,n}$ is at least $N$. Then for every $k < N$ we have
\begin{gather*}
|\lambda_{r,n,k}| \leq |\lambda_{r,n,0}| + 2 r k.
\end{gather*}
Recall the definition \eqref{eq: xi function def} of the function $\xi$. The previous inequality implies
\begin{gather*}
\xi(|\lambda_{n_0}|)\ge\sum_{k=0}^{N - 1} \frac{1}{|\lambda_{n_k}|} \geq \sum_{k=0}^{N - 1} \frac{1}{|\lambda_{n_1}| + 2rk} \ge\frac{1}{2r} \log \left(1 +\frac{2rN}{|\lambda_{n_0}|}\right).
\end{gather*}
The uniform in $n$ estimate  $\#\Lambda_{r,n} \ls \xi(|\lambda_{r,n,0}|)|\lambda_{r,n,0}|$ follows.
This also shows
\begin{gather}
\label{eq: w estimate in lambda}
|w| \ls |\lambda_{r,n,0}| + 1, \qquad w\in U_{r,n},\qquad n \ge 1.
\end{gather}
The set $U_{r,n}$ is a finite union of circles hence its boundary $\Gamma_{r,n}$ is a subset of a finite number (not greater than $\#\Lambda_{r,n}$) of circles. It follows that
\begin{gather}
    \label{eq: contour estimate}
    \len(\Gamma_{r,n}) \le 2\pi r\cdot \# \Lambda_{r,n}\ls \xi(|\lambda_{r,n,0}|)|\lambda_{r,n,0}| \ls |\lambda_{r,n,0}|,\qquad n\ge 1.
\end{gather}

\subsection{Proof of Theorem \ref{thm: wider domain}}
Let us first proof the theorem under the additional assumption 
\begin{gather}
    \label{eq: assumption large Im}
    \Im \lambda_n \le -2, \qquad n\ge 1.
\end{gather}
We proceed as was described in Section \ref{section: idea}. We need to solve the interpolation problem \eqref{eq: interpolation problem idea}, 
\begin{gather*}
    g\in \PW_\sigma, \qquad g(\lambda) = G(\lambda),\qquad \lambda\in  \Lambda\cup\{0\},
\end{gather*}
where $G(z) = z e^{-i\sigma z}$. Let $U_n = U_{1, n}, \Lambda_n = \Lambda_{n, 1} = \{\lambda_{n, 0}, \lambda_{n, 1}, \ldots\}$ and $\Gamma_n = \Gamma_{n,1}$ be as in the previous section.
For every $n\ge 0$ we define the meromorphic function
\begin{gather*}
    F_n(z) = \sum_{\lambda\in \Lambda_n}\frac{G(\lambda)}{H'(\lambda)(z - \lambda)} 
\end{gather*}
with simple poles in $\Lambda_n$ and rewrite it using relation \eqref{eq: from sum to contour}:
\begin{gather}
\label{eq: f_n and contour again}
    F_n(z) = \frac{1}{2\pi i}\oint_{\Gamma_n}\frac{G(w)\, dw}{H(w)(z - w)}, \qquad z\notin U_n.
\end{gather}
By the construction we have $\dist(w, \Lambda) = \dist(w, \Lambda_n) = 1$ for every $w\in \Gamma_n$. Assumption \eqref{eq: assumption large Im} then implies
\begin{gather*}
    \dist(w, -\Lambda) \ge 1, \qquad |\Im w|\ge 1,\qquad w\in \Gamma_n.
\end{gather*}
Theorem \ref{thm: H lower bound} applies. It gives
\begin{gather*}
    \log|H(w)| \gs -  |w|^{\alpha}\log |w|,\qquad w\in \Gamma_n,\qquad n\ge 0.     
\end{gather*}
Due to the assumption $\Lambda\subset \{z: -\Im z \geq C | z|^\beta\}$ we also know that
\begin{gather*}
    |G(w)| \ls |w|\exp(-C_1|w|^{\beta})
\end{gather*}
holds with some other constant $C_1 > 0$ for all $w\in \Gamma_n$.
Substituting the two latter inequalities into \eqref{eq: f_n and contour again}, we get
\begin{gather*}
    |F_n(z)|\ls \frac{\len(\Gamma_n)}{\dist(z, \Gamma_n)} \sup_{w\in \Gamma_n}e^{-C_2|w|^{\beta}}. 
\end{gather*}
From estimates \eqref{eq: w estimate in lambda} and \eqref{eq: contour estimate} we know the relations $\len(\Gamma_n)\ls |\lambda_{n,0}|$ and $|w|\asymp |\lambda_{n, 0}|$ for all $w\in \Gamma_n$. It follows that the series 
\begin{gather*}
    g(z) = \sum_{n\ge 1} H(z)F_n(z), \qquad z\in \Cm,
\end{gather*}
converges for all $z\in \Cm$ and defines an entire function $g$ that satisfies $|g(z)|\ls |H(z)|$ when $\dist(z, \Gamma_n) > 1$ for all $n\ge 0$. Notice that
\begin{gather*}
    \bigcup_{n\ge 0}\big\{z\colon \dist(z, \Gamma_n)\le 1\big\} = \big\{z\colon \dist(z, \Lambda)\le 2\big\} = V_2,
\end{gather*}
where $V_2$ is defined in \eqref{eq: Vr def}. In the previous section we have shown that $V_2$ has a sparse structure, recall Figure \ref{fig:cluster}.
Maximum modulus principle and the Paley-Wiener theorem give $g\in \PW_\sigma$.  
To finish the proof of the theorem we need to fulfill the assertion $g(0) = 0$ and drop \eqref{eq: assumption large Im}. At this point we proceed as in the proof of Theorem \ref{thm: angle}: construct $\tilde g$ for the set 
\begin{gather*}
    \tilde\Lambda = \{-2i\}\cup\{\lambda_n - 2i,\, n\ge 1\}. 
\end{gather*}
as described above and take $g = \tilde g(\cdot + 2i)$.

\section{Sharpness of the results. Proof of Theorem \ref{thm non res example}}
Resonances of the \Schr operator can be defined as the zeroes of the corresponding de Branges function $E = A + iB$, see \cite{baranov2017branges}. The following lemma is stated in Lemma 3 in the preprint \cite{baranov2015branges}. We reproduce its proof here for the reader's convenience.
\begin{Lem}\label{Lem: phase function}
For every compactly supported potential $q\in L^1(\R_+)$ we have
    \begin{gather*} 
    \sup_{t\in \R} \sum_{s \in \Res(q)} \frac{|\Im s|}{|t-s|^2}  < \infty.
\end{gather*}
\end{Lem}
\begin{proof}
The phase function of the \Schr operator is defined by the relation $E/\ol{E} = e^{-2i\phi}$, see \cite{baranov2017branges}, it satisfies
\begin{gather*}
    \phi'(t) = \sum_{s \in \Res(q)} \frac{|\Im s|}{|t-s|^2} + a,
\end{gather*}
where $a$ is a real-valued constant. The calculation shows 
    \begin{align}
    \nonumber
        -2i\phi' &= \frac{E'}{E} + \frac{\ol{E'}}{\ol{E}} = \frac{E'\ol{E} - E\ol{E}'}{|E|^2} 
        \\
        \label{eq: phi as A/B}
        &= \frac{AB' - A'B}{A^2 + B^2} = \frac{AB' - A'B}{B^2}\cdot \frac{1}{(A/B)^2 + 1} = \left(\frac{A}{B}\right)'\frac{1}{(A/B)^2 + 1}.
    \end{align}
Let $\{\lambda_n\}_{n\in \Z}$ and $\{\mu_n\}_{n\in \Z}$ be the zeroes of $A$ and $B$ respectively. Lemma 1 in \cite{baranov2017branges} states that for $q\in L^2(\R_+)$ the functions $A$ and $B$ satisfy
\begin{gather}
    \label{eq: A and B asymp}
    \left|\frac{A(z)}{\sin z}\right|\asymp \frac{\dist(z, \{\lambda_n\})}{\dist(z, \pi\Z)}, \qquad \left|\frac{B(z)}{\cos z}\right|\asymp \frac{\dist(z, \{\mu_n\})}{\dist(z, \pi\Z + \pi/2)}.
\end{gather}
We claim that these relations hold for all $q\in L^1(\R_+)$. In \cite{baranov2017branges} the $L^2$ assertion is used in the asymptotic formula (3.3) for $\lambda_n$ and $\mu_n$,
\begin{gather*}
    \lambda_n = \pi n + \frac{C}{n} + \frac{a_n}{n},\qquad \mu_n = \pi_n - \frac{\pi}{2} + \frac{C}{n} + \frac{b_n}{n},\qquad\{a_n\}, \{b_n\}\in \ell^2(\Z_+) \qquad n\ge 1.
\end{gather*}
For $q\in L^1(\R_+)$ the latter holds with $\{a_n\}, \{b_n\}\in \ell^\infty(\Z_+)$, see, e.g., \cite{Chelkak2010}, which is sufficient for the rest of the proof.

Substitute $z = \mu_k$ into \eqref{eq: A and B asymp}. We have
\begin{gather*}
    |A(\mu_k)|\asymp |\sin(\mu_k)|\frac{\dist(\mu_k, \{\lambda_n\})}{\dist(\mu_k, \pi\Z)} \asymp 1,\qquad |B'(\mu_k)|\asymp 1.
\end{gather*}
For every $n\in \Z$ let $v_n = A(\mu_n)/B'(\mu_n)$. This sequence is bounded and we can write 
\begin{gather*}
    \frac{A}{B} = \sum_{n\in \Z}\frac{A(\mu_n)}{B'(\mu_n)}\left(\frac{1}{z - \mu_n} + \frac{1}{\mu_n}\right) = \sum_{n\in \Z}v_n\left(\frac{1}{z - \mu_n} + \frac{1}{\mu_n}\right),
    \\
    \left(\frac{A}{B} \right)' = \sum_{n\in \Z}v_n\frac{1}{(z - \mu_n)^2}.
\end{gather*}
Inequality \eqref{eq: phi as A/B} implies 
\begin{gather*}
    |\phi'(x)|\ls \sum_{n\in \Z}\frac{1}{(x - \mu_n)^2} \cdot \Bigg|1 + \bigg(\sum_{n\in \Z}v_n\left(\frac{1}{z - \mu_n} + \frac{1}{\mu_n}\right)\bigg)^2\Bigg|^{-1}.
\end{gather*}
The second multiplier is always less or equal to $1$ 
hence we get
\begin{gather*}
    |\phi'(x)|\le \sum_{n\in \Z}\frac{1}{(x - \mu_n)^2}\ls 1.
\end{gather*}
Fix small $\eps > 0$.  If $\dist(x, \{\mu_n\})\ge \eps$ this gives $|\phi'(x)|\ls 1$. Otherwise there exists $\mu_k$ such that $|x - \mu_k| = \dist(x, \{\mu_n\})\le \eps$. In this case we have
\begin{gather*}
    \sum_{n\in \Z}\frac{1}{(x - \mu_n)^2}\asymp \frac{1}{(x - \mu_k)^2},
    \qquad 
    \left|\sum_{n\in \Z}v_n\left(\frac{1}{z - \mu_n} + \frac{1}{\mu_n}\right)\right|\asymp \frac{1}{|z - \mu_k|},
\end{gather*}
which again implies $|\phi'(x)|\ls 1$. The proof is concluded.
\end{proof}

\begin{proof}[Proof of Theorem \ref{thm non res example}]
First of all we notice that it is sufficient to prove theorem only when $\rho$ satisfies $\rho(r) = o(r)$ as $r\to\infty$. Indeed, for every compactly supported $q\in L^1(\R_+)$, the counting function $N$ of $\Res(q)$ satisfies $N(r) = O(r)$ as $r\to\infty$ hence we may assume $\rho(r) = O(r)$ and $\tau(r) = o(\rho(r)) = o(r)$ as $r\to\infty$. Next, if $\rho(r) = o(r)$ does not hold we can prove the theorem with $\tilde\rho(r) = \sqrt{\rho(r)\tau(r)}$ instead of $\rho$.

Consider another function $ \kappa(r) = \frac{1}{2}(\rho(r) + \tau(r)) $ 
and define $r_n$ as the solution of $ \kappa(r_n) = n $ for all $n\ge \kappa(0)$. The assertion $\kappa(r) = o(r)$ as $r\to\infty$ implies
\begin{gather}
    \label{eq: n/r_n limit}
    \lim_{n \to \infty} \frac{n}{r_n} = \lim_{n \to \infty} \frac{\kappa(r_n)}{r_n} = 0.
\end{gather}
We also have $\tau(r) = o(r)$ as $r\to\infty$ hence for all large $n$ there exists a unique $z_n\in\Cm$ that satisfies
\begin{gather*}
    |z_n| = r_n, \qquad \Im z_n = - \tau(r_n), \qquad \Re z_n > 0.
\end{gather*}
We construct an increasing sequence $\{n_0, n_1, \ldots\}$ inductively. Choose $n_0 $ for which $ z_{n_0} $ is well-defined.
Next, for every $ k \geq 1 $ select  $ n_k > n_{k - 1} $ such that $ 2^k n_k \leq r_{n_k} $ and
\begin{gather*}
n_0 + \cdots + n_{k-1} \le \rho(r_{n_k}) - \kappa(r_{n_k}).
\end{gather*}
Such $n_k$ exists for every $k\ge 1$ because of \eqref{eq: n/r_n limit} and
\begin{gather*}
\lim_{n \to \infty} \left( \rho(r_n) - \kappa(r_n) \right) = \frac{1}{2} \lim_{n \to \infty} \left( \rho(r_n) - \tau(r_n) \right) = +\infty.
\end{gather*}
We form the multiset $\Lambda = \{ z_{n_k} \colon k \geq 0 \}$, where $z_{n_k}$ appears with multiplicity $n_k$. Equivalently, we can replace each instance of $z_{n_k}$ with $n_k$ distinct points arbitrarily close to $z_{n_k}$. We have
\begin{gather*}
\sum_{\lambda \in \Lambda} \frac{1}{|\lambda|} = \sum_{k=0}^{\infty} \frac{n_k}{|z_{n_k}|} = \sum_{k=0}^{\infty} \frac{n_k}{r_{n_k}} \leq \sum_{k=0}^{\infty} \frac{1}{2^k} < \infty.
\end{gather*}
Furthermore, we can estimate the counting function as follows:
\begin{gather*}
n(r_{n_k}) = n_0 + \cdots + n_k \le \rho(r_{n_k}) - \kappa(r_{n_k}) + n_k = \rho(r_{n_k}).
\end{gather*}
The inequality $ n(r) \leq \rho(r) $ for all $ r \geq 0 $ follows.
\medskip 

Assume that $ \Lambda \subset \Res(q) $ for some compactly supported $ q \in L^1(\R_+) $. By Lemma \ref{Lem: phase function} the function
\begin{gather*}
p(t) = \sum_{s \in \Lambda} \frac{|\Im s|}{|t-s|^2}
\end{gather*}
is uniformly bounded for $ t \in \R$. On the other hand, we have
\begin{gather*}
p(\Re z_{n_k}) = \sum_{s \in \Lambda} \frac{|\Im s|}{|\Re z_{n_k} -s|^2} \geq \frac{n_k |\Im z_{n_k}|}{|\Re z_{n_k} - z_{n_k}|^2} = \frac{n_k}{|\Im z_{n_k}|} = \frac{\kappa(r_{n_k})}{\tau(r_{n_k})}\to\infty,\qquad n\to\infty.
\end{gather*}
This contradiction shows that $ \Lambda $ is not a subset of $ \Res(q) $.
\end{proof}

\bibliographystyle{plain} 
\bibliography{ref}

\end{document}